\documentclass[preprint,12pt]{elsarticle}
\usepackage{epsfig}
\usepackage{algorithmicx}
\usepackage[ruled]{algorithm}
\usepackage{algpseudocode}
\usepackage{amsmath}
\usepackage{amssymb}
\usepackage{subfig}
\usepackage{xspace}
\usepackage{tabularx}
\usepackage{url}
\usepackage{multirow}
\usepackage{amsthm}
\newtheorem{theorem}{\bf Theorem}[section]

\newproof{pf}{Proof}

\newtheorem{example}{\bf Example}[section]

\begin{document}

\begin{frontmatter}



\title{A Systematic Framework for Stable and Cost-Efficient Matrix Polynomial Evaluation}


\author[I3M]{J.M. Alonso}
\ead{jmalonso@dsic.upv.es}
\author[ITEAM]{J. Sastre}
\ead{jsastrem@upv.es}
\author[IMM]{J. Ib\'a\~nez}
\ead{jjibanez@dsic.upv.es}
\author[IMM]{E. Defez}
\ead{edefez@imm.upv.es}

\address[I3M]{Instituto Universitario de Instrumentaci\'on para Imagen Molecular,
Universitat Polit\`ecnica de Val\`encia, Camino de Vera s/n,
46022-Valencia (Spain)}
\address[ITEAM]{Instituto de Telecomunicaciones y Aplicaciones Multimedia,
Universitat Polit\`ecnica de Val\`encia, Camino de Vera s/n, 46022-Valencia (Spain)}
\address[IMM]{Instituto Universitario de Matemática Multidisciplinar,
Universitat Polit\`ecnica de Val\`encia, Camino de Vera s/n, 46022-Valencia (Spain)}

\begin{abstract}
A method for evaluating matrix polynomials have recently been
developed that require one fewer matrix product ($1M$) than the
Paterson--Stockmeyer (PS) method. Since the computational cost for
large-scale matrices is asymptotically determined by the number of
matrix products, this reduction directly affects the total execution
time. However, the coefficients in these optimized formulas emerge
as solutions to systems of nonlinear polynomial equations, resulting
in multiple potential solution sets. An inappropriate selection of
these coefficients can lead to numerical instability in
floating-point arithmetic.

This paper presents a systematic framework and a MATLAB
implementation, \texttt{MatrixPolEval1}, used to obtain and validate
stable coefficient sets for matrix polynomials of degrees $m \in
\{8, 10, 12\}$ and above. The framework introduces structural
variants to maintain stability even when the original configuration
fails to yield a robust solution. The provided tool identifies
stable coefficient sets using variable precision arithmetic (VPA)
and provides a reliability indicator for expected accuracy.
Numerical experiments on polynomials arising in applications,
including the matrix exponential and geometric series, show that the
framework achieves the $1M$ saving while maintaining numerical
accuracy comparable to the PS method.
\end{abstract}


\end{frontmatter}

\section{Introduction}\label{section_introduction}

The evaluation of a matrix polynomial
\begin{equation}\label{P}
    P_{m}(A)=\sum_{i=0}^{m} b_{i} A^{i}
\end{equation}
can be performed with reduced computational cost using the schemes
proposed in \cite{Sastre18}, which utilize products and linear
combinations of matrix polynomials. Specifically, \cite[Prop.
1]{Sastre18} shows that for general matrix polynomials of certain
degrees with non-zero leading coefficients, it is possible to reduce
the computational cost by one matrix product ($1M$) compared to the
Paterson--Stockmeyer (PS) method \cite{PaSt73}. In this context, the
number of matrix--matrix multiplications is used as the primary cost
metric, as other operations become asymptotically negligible for
large matrices. More recently, the theoretical lower bounds for
these evaluations were further characterized in \cite{JL25}.
Furthermore, the specific case of degree-20 polynomials arising from
Conjecture 5.1 in \cite{JL25} has been addressed in \cite{SIAD25}.

However, the coefficients in the evaluation formulas of
\cite{Sastre18} are derived from the roots of nonlinear polynomial
equations, resulting in multiple potential solution sets. As shown
in \cite[Sec. 3.1]{Sastre18}, the numerical stability of these sets
varies, making the identification of stable solutions a requirement
for maintaining precision. While stable solutions have been
identified for specific cases, such as Taylor approximations of the
matrix exponential \cite{SID19} and the matrix cosine
\cite{SIAPD19}, a general methodology for arbitrary polynomials has
yet to be established. Related schemes for specific approximation
degrees of the matrix exponential, sine, and cosine have also been
developed in \cite{BBS19, SBC21, BBSC22}.

This paper presents a systematic framework and a MATLAB
implementation, \texttt{MatrixPolEval1}, used to calculate and
validate stable coefficient sets based on \cite[Prop. 1]{Sastre18}
for matrix polynomials of degrees $m = 8, 10, 12$ and above. The
tool automates the selection process to identify the most stable
solution from the generated sets. Additionally, for cases where the
formulas in \cite[Prop. 1]{Sastre18} do not yield robust results,
alternative structural variants are provided. The efficiency and
numerical stability of the proposed framework are illustrated
through experiments arising in applications, including a comparison
with the PS method.

Throughout this paper, $\mathbb{N}$ denotes the set of natural
numbers $\{0, 1, 2, \dots\}$. For $x \in \mathbb{R}$, $\lceil x
\rceil$ denotes the smallest integer greater than or equal to $x$,
and $\lfloor x \rfloor$ denotes the largest integer less than or
equal to $x$.

This paper is organized as follows. Section \ref{sec:PS} briefly
reviews the PS method. Section \ref{sec:efficienteval} extends the
methodology from \cite{Sastre18} to provide a general framework for
matrix polynomial evaluation using one fewer matrix product than the
PS method. Section \ref{sec:Implementation} describes the
algorithmic implementation, including the treatment of structural
variants and numerical stability. Numerical experiments are
presented in Section \ref{sec:numericalresults}, and concluding
remarks are provided in Section \ref{sec:conclusions}.

\section{The Paterson--Stockmeyer matrix polynomial evaluation method}\label{sec:PS}

The PS evaluation scheme for matrix polynomials \eqref{P}, denoted
by $PS_{m}(A)$, is defined in \cite[Algorithm B]{PaSt73} as:
\begin{equation}\label{PPS}
\begin{aligned}
    PS_{m}(A) = & \left( \dots \left( \left( \sum_{j=0}^{s} b_{m-j} A^{s-j} \right) A^s + \sum_{j=1}^{s} b_{m-s-j} A^{s-j} \right) A^s + \dots \right) A^s \\
    & + \sum_{j=1}^{s} b_{s-j} A^{s-j},
\end{aligned}
\end{equation}
where $s \in \mathbb{N} \setminus \{0\}$ and $m$ is a multiple of
$s$. If $m$ is not a multiple of $s$, the formula \eqref{PPS}
remains applicable by evaluating $PS_{m_0}(A)$ with $m_0 = s \lceil
m/s \rceil$ and setting $b_k = 0$ for $k = m+1, \dots, m_0$.

For a given number of matrix products, the maximum degrees
achievable by the PS method, denoted by $m^*$, are given by:
\begin{equation}\label{eq_optimalm}
    m^* = s^2 \quad \text{or} \quad m^* = s(s+1),
\end{equation}
where $s \in \mathbb{N} \setminus \{0\}$. This generates the
sequence of optimal degrees $\mathcal{M} = \{1, 2, 4, 6, 9, 12,
\dots\}$ \cite[Sec. 2.1]{Sastre12}. Furthermore, \cite{Fasi19}
demonstrated the optimality of the rule $m^* = (C_{PS}-s+2)s$, where
$C_{PS}$ denotes the number of matrix products required by the PS
method and $s = \lfloor C_{PS}/2 \rfloor + 1$. This formulation
yields maximum degrees consistent with \eqref{eq_optimalm}
\cite[Sec. 2.1]{SaIb21}.

If $m \notin \mathcal{M}$, the polynomial $P_{m}(A)$ is evaluated
using the $PS_{m_0}(A)$ scheme by choosing $m_0 = \min \{ k \in
\mathcal{M} : k > m \}$ and setting the additional leading
coefficients to zero ($b_i = 0$ for $i = m+1, \dots, m_0$)
\cite[Sec. 2.1]{Sastre18}. Consequently, the minimum computational
cost of the PS method is:
\begin{equation}\label{PScost}
    C_{PS} = s + \left\lceil \frac{m}{s} \right\rceil - 2,
\end{equation}
where the optimal value of $s$ is either $\lfloor \sqrt{m} \rfloor$
or $\lceil \sqrt{m} \rceil$.

\section{Matrix polynomial evaluation beyond the PS method}\label{sec:efficienteval}

The methodology proposed in \cite{Sastre18} provides a framework to
evaluate matrix polynomials with reduced computational cost compared
to the PS method. This section details the evaluation of $P_8(A)$,
where the cost is reduced to $3M$ (one product fewer than the PS
method), and extends the framework to polynomials of degree $m > 8$.

\subsection{Coefficient solutions for $P_8(A)$ evaluation with $3M$}
To evaluate $P_8(A)$ with $3M$, the following nested evaluation
scheme is used \cite[Ex. 3.1]{Sastre18}:
\begin{equation}\label{m8s2by}
\begin{aligned}
    y_{02}(A) &= A^2(c_4A^2 + c_3A), \\
    y_{12}(A) &= (y_{02}(A) + d_2A^2 + d_1A)(y_{02}(A) + e_2A^2) \\
              &\quad + e_0y_{02}(A) + f_2A^2 + f_1A + f_0I.
\end{aligned}
\end{equation}By equating $y_{12}(A) = P_8(A)=\sum_0^8 b_iA^i$, a system of
nonlinear equations for the scalar coefficients $\{c_i, d_i, e_i,
f_i\}$ arises. While coefficients such as $f_{0,1,2}$ and $c_{3,4}$
are obtained directly from $b_i$, the remaining parameters depend on
the solution of a quadratic equation for $e_2$ \cite[Eq.
31]{Sastre18}:
\begin{equation}\label{m8s2bsystemsole2}
e_2 = \frac{\frac{c_3}{c_4}\underline{de}_2 - d_1 \pm
\sqrt{\left(d_1 - \frac{c_3}{c_4}\underline{de}_2\right)^2 +
4\frac{c_3}{c_4}\left(b_3 + \frac{c_3^2}{c_4}d_1 -
\frac{c_3}{c_4}b_4\right)}}{2c_3/c_4},
\end{equation}where $\underline{de}_2 = d_2 + e_2$. This quantity can be computed as
$\underline{de}_2 = (b_6 - c_3^2)/c_4$ using \cite[Eq.
27]{Sastre18}.

This expression yields up to four potential solution
sets depending on the signs of  $c_4=\pm\sqrt{b_8}$ \cite[Ex.
3.1]{Sastre18}. A fundamental requirement for this scheme is $b_8
\neq 0$. If $b_8 < 0$, the framework evaluates $y_{12}(A) =
-P_m(A)$. This prevents $c_4$ from becoming a complex value, thereby
avoiding complex arithmetic when $A$ is a real matrix.

\subsubsection{Handling the singular case $c_3=0$}
The original analysis in \cite{Sastre18} did not address the case
$c_3 = 0$ (occurring when $b_7 = 0$), which leads to an
indeterminate form in \eqref{m8s2bsystemsole2}. Our framework
identifies that in this scenario, the system reduces to a linear
form where
\begin{equation}\label{c3null_sol}
\begin{aligned}
    c_4 &= \pm\sqrt{b_8}, \quad d_1 = b_5/c_4, \quad e_2 = b_3/d_1, \\
    d_2 &= b_6/c_4 - e_2, \quad e_0 = (b_4 - d_2e_2)/c_4.
\end{aligned}
\end{equation}
This specific branch requires the additional condition $b_5 \neq 0$
to ensure $e_2$ is well-defined.

\subsection{Numerical stability and coefficient selection}
The existence of multiple solution sets necessitates a selection
criterion based on numerical stability. Following \cite{Sastre18},
we assess stability by computing the relative error $er_i$ for each
reconstructed coefficient $b_i$ when using finite-precision
arithmetic. If $\tilde{x}$ denotes the value of $x$ rounded to the
target precision, the error for $b_8$ is:
\begin{equation}\label{stability}
    er_{8} = \frac{|b_8 - \tilde{c}^2_4|}{|b_8|}.
\end{equation}
A solution set is considered stable if all $er_i$ are of the order
of the unit roundoff $u$ in the target precision (for instance, $u =
2^{-24} \approx 5.96 \times 10^{-8}$ for IEEE single precision and
$u = 2^{-53} \approx 1.11 \times 10^{-16}$ for IEEE double
precision). To ensure robustness, the framework extends this check
by using absolute errors when $b_i = 0$. This prevents division by
zero and facilitates the automated selection of the most accurate
evaluation scheme.

\subsection{General evaluation formula for $P_{m}(A)$ for $m \ge 8$}

In \cite[Sec. 3]{Sastre18}, a nested evaluation scheme was
introduced to compute matrix polynomials $P_m(A)$ of degree $m=4s$
($s \ge 2$) using $s+1$ matrix products. The structure is defined
as:
\begin{subequations}\label{eq:msy_original}
\begin{align}
y_{0s}(A) &= A^s\sum_{i=1}^{s}c_{s+i}A^i, \label{msy0order1} \\
y_{1s}(A) &= \left(y_{0s}(A)+\sum_{i=1}^{s}d_{i}A^i\right)\left(y_{0s}(A)+\sum_{i=2}^{s}e_{i}A^i\right) \nonumber \\
&\quad +e_{0}y_{0s}(A)+\sum_{i=0}^{s}f_{i}A^i, \label{msy1order1}
\end{align}
\end{subequations}where $c_{s+i}$, $d_i$, $e_i$, and $f_i$ are scalar coefficients.
The powers $A^i$ for $i=2,3,\dots,s$ are computed once and stored.
The total cost of evaluating $y_{1s}(A)$ is $C_{y_{1s}}=s+1$ matrix
products ($M$).

Equating $y_{1s}(A)=P_m(A)$ results in a system of $4s+1$ equations.
Provided that the following conditions are fulfilled:
\begin{subequations}\label{eq:conditions}
\begin{align}
b_{4s} &\neq 0, \label{c4sneq0} \\
d_s &\neq e_s \quad \text{for } s > 2, \label{ds_neq_es}
\end{align}
\end{subequations}the system reduces to finding the roots of a polynomial in the
variable $e_s$. Again, if the leading coefficient $b_{4s}$ is
negative, the framework equates $y_{1s}(A) = -P_m(A)$. This ensures
that the coefficient $c_{2s} = \pm\sqrt{b_{4s}}$ remains real,
thereby avoiding complex arithmetic when $A$ is a real matrix
\cite[Sec. 3]{Sastre18}.

For polynomials of degree $m \ge 8$, the scheme \eqref{msy1order1}
is integrated into the PS-like structure $z_{1ps}(A)$ \cite[Eq.
52]{Sastre18}. Let $m = 4s + p$, where $p \ge 0$. The generalized
formula $z_{1ps}(A)$ is defined as:
\begin{equation}\label{mixps}
\begin{aligned}
    z_{1ps}(A) &= \left( \dots \left( \left( y_{1s}(A)A^s + \sum_{i=1}^{s} a_{p-i} A^{s-i} \right) A^s + \sum_{i=1}^{s} a_{p-s-i} A^{s-i} \right) A^s + \dots \right. \\
    &\quad \times \left. A^s + \sum_{i=1}^{s} a_{p-(t-1)s-i} A^{s-i} \right) A^{r} + \sum_{j=0}^{r-1} a_{j} A^{j},
\end{aligned}
\end{equation}where $t = \lfloor p/s \rfloor$ is the number of full blocks of
degree $s$, and $r = p \bmod s$ is the degree of the remainder block,
such that $s > r \ge 0$.

\begin{theorem}\label{thm:cost}
Let $P_m(A)$ be a matrix polynomial of degree $m \ge 8$. For $s \ge
2$ and $m = 4s + p$ ($p \ge 0$), $P_m(A)$ can be evaluated using
\eqref{mixps} with a cost of
\begin{equation}\label{costz1s}
C_{z_{1ps}}(m) = (s + 1 + \lceil p/s \rceil) M.
\end{equation}Furthermore, for any $m \ge 8$ (excluding $m=9, 11$),
there exists a value of $s$ such that $C_{z_{1ps}}(m) =
C_{PS}(m)-1$, where $C_{PS}(m)$ is the cost of the optimal PS
method.
\end{theorem}

\begin{proof}
The cost $s+1$ accounts for the evaluation cost of $y_{0s}(A)$ from
\eqref{msy0order1}, $y_{1s}(A)$ from \eqref{msy1order1} and the
initial powers $A^2,\,A^3,\ldots\,A^s$. The term $\lceil p/s \rceil$
accounts for the $t$ multiplications by $A^s$ in the nesting levels
of \eqref{mixps} and, if $r > 0$, one final multiplication by $A^r$.

Regarding the efficiency compared to the PS method, whose cost is
given by \eqref{PScost}, we analyze the available degrees $m$:
\begin{itemize}
\item For $s=2$, the formula covers $m = 8 + p$:
\begin{itemize}
    \item $m=8$ ($p=0$): $C_{z_{1ps}} = 3M$, while $C_{PS} = 4M$.
    \item $m=9$ ($p=1$): $C_{z_{1ps}} = 4M$. The optimal PS cost is also $4M$ (using $s=3$ in \eqref{PPS}); thus, this method does not provide a reduction in this specific case.
    \item $m=10$ ($p=2$): $C_{z_{1ps}} = 4M$, while $C_{PS} = 5M$.
    \item $m=11$ ($p=3$): $C_{z_{1ps}} = 5M$. Since the optimal PS cost for $m=11$ is also $5M$ (e.g., using $s=3$ or $s=4$), no computational saving is achieved for this degree.
\end{itemize}
    \item For $m \ge 12$, any degree can be reached by selecting an appropriate $s$ and $p$, consistently achieving a cost reduction of
    $1M$:
    \begin{itemize}
        \item For $s=3$, the range $m \in [12, 15]$ is covered by varying $p \in \{0, 1, 2, 3\}$.
        \item For $m=16$, we shift to $s=4$ ($4s=16$).
    \end{itemize}
\end{itemize}

In general, for any $s \ge 3$, the range of degrees $[4s, 4s+s]$ is
fully covered. Since the starting point of the next range, $4(s+1)$,
is always less than or equal to the end of the current range, $5s$,
for all $s \ge 4$ (and specifically $4(3)+3=15$ is followed by
$4(4)=16$), there are no gaps in the available degrees for $m \ge
12$. In all these cases, the cost is reduced by one matrix product
compared to the optimal PS cost.
\end{proof}

The practical application of formula \eqref{mixps} depends on
numerical stability. As noted in \cite[Ex. 3.2]{Sastre18},
coefficient selection is required to maintain accuracy. For example,
evaluating $\exp(1)$ with $m=28$ and $s=7$ using an unstable
coefficient set results in a relative error of $9.81 \times
10^{-5}$, whereas the most stable set identified by the proposed
framework yields $8.63 \times 10^{-17} < u$, as shown in Section
\ref{sec:numericalresults}.

If no stable solution is found for a specific pair $(s, p)$, the
framework allows to explore alternative decompositions for the same
$m$. For $m=28$, this includes $(s=4, p=12)$, $(s=5, p=8)$, or
$(s=6, p=4)$. For the specific case $m=12$, an alternative formula
provided in \cite[Sec. 5.1]{JL25} avoids square root calculations
and requires only a non-zero leading coefficient. The requirement
for a square root of the leading coefficient in \eqref{msy0order1}
and \eqref{msy1order1} can also be circumvented by reformulating the
scheme as:
\begin{subequations}\label{eq:msy_nosqrt}
\begin{align}
y_{0s}(A) &= A^s \left( A^s + \sum_{i=1}^{s-1} c_{s+i} A^i \right), \label{msy0order1nosqrt} \\
y_{1s}(A) &= b_{4s} \left( y_{0s}(A) + \sum_{i=1}^{s} d_{i} A^i \right) \left( y_{0s}(A) + \sum_{i=2}^{s} e_{i} A^i \right) \nonumber \\
&\quad + e_{0} y_{0s}(A) + \sum_{i=0}^{s} f_{i} A^i,
\label{msy1order1nosqrt}
\end{align}
\end{subequations}where $b_{4s} \neq 0$ is the leading coefficient of $P_m(A)$. The
combination of the nested structures \eqref{msy0order1} and
\eqref{msy1order1} with \eqref{mixps} generates multiple solution
sets. This redundancy allows for the automated selection of the most
stable solution, thereby improving the numerical accuracy of the
evaluation.

\section{Numerical implementation and solver selection}\label{sec:Implementation}

The proposed methodology is summarized in Algorithm
\ref{MatrixPolEval1}. As demonstrated in \cite[pp.
237--240]{Sastre18}, the system of $4s+1$ nonlinear equations
derived by equating the coefficients of $y_{1s}(A)$ and $P_m(A)$ can
be reduced via variable substitution to computing the roots of a
univariate polynomial. For this reason, the implementation relies on
the \texttt{vpasolve} function from the MATLAB Symbolic Math Toolbox
\cite{SymbolicToolbox}.

Unlike standard numerical solvers that converge to a single local
solution, \texttt{vpasolve} utilizes variable precision arithmetic
(VPA) to search for multiple sets of roots. This is a critical
feature for the proposed algorithm, as it allows for the exploration
of the solution space to identify the set of coefficients that
exhibits the highest numerical stability. Furthermore, working with
VPA and increased significant digits during the solving process
mitigates the accumulation of rounding errors that typically occur
when dealing with ill-conditioned polynomial systems. While the
symbolic computation may entail a higher initial cost, this
calculation is performed offline. Once the stable coefficients are
obtained, they are stored and used for efficient matrix polynomial
evaluations.

To provide more flexibility, the algorithm implements three
structural variants for $y_{1s}(A)$ based on the base level
\begin{equation}\label{msy0_final}
    y_{0s}(A) = A^s \sum_{i=1}^{s} c_{m+1-s+i} A^i,
\end{equation}
where the coefficients $c$ are indexed correlatively to facilitate
numerical implementation and ensure consistency with the provided
MATLAB code. Selected via the \texttt{type\_pol} parameter detailed
below, these variants expand the search space for stable real
solutions and are defined as follows:

\begin{enumerate}
    \item Type 1 (corresponding to the original \eqref{msy1order1}):
    \begin{equation}\label{type1_math}
    \begin{split}
        y_{1s}(A) &= \left(y_{0s}(A) + \sum_{i=1}^{s} c_{m+2-2s+i-1} A^i \right) \\
        &\quad \times \left(y_{0s}(A) + \sum_{i=2}^{s} c_{m+3-3s+i-2} A^i \right) \\
        &\quad + c_{m+2-3s}y_{0s}(A) + \sum_{i=0}^{s} c_{m+1-4s+i} A^i
    \end{split}
    \end{equation}

    \item Type 2:
    \begin{equation}\label{type2_math}
    \begin{split}
        y_{1s}(A) &= \left(y_{0s}(A) + \sum_{i=0}^{s} c_{m+1-2s+i} A^i \right) \\
        &\quad \times \left(y_{0s}(A) + \sum_{i=2}^{s} c_{m+2-3s+i-2} A^i \right) \\
        &\quad + \sum_{i=0}^{s} c_{m+1-4s+i} A^i
    \end{split}
    \end{equation}

    \item Type 3:
    \begin{equation}\label{type3_math}
    \begin{split}
        y_{1s}(A) &= \left(y_{0s}(A) + \sum_{i=1}^{s} c_{m+2-2s+i-1} A^i \right) \\
        &\quad \times \left(y_{0s}(A) + \sum_{i=1}^{s} c_{m+2-3s+i-1} A^i \right) \\
        &\quad + \sum_{i=0}^{s} c_{m+1-4s+i} A^i
    \end{split}
    \end{equation}
\end{enumerate}

The selection of the optimal coefficient set among the multiple
solutions is based on a stability metric $\epsilon_{max}$. While
previous approaches \cite{Sastre18, SID19, SIAPD19} typically rely
on a relative error criterion, our framework extends the stability
check to ensure robustness for polynomials with vanishing
coefficients. Specifically, the error $\delta_i$ for each
reconstructed coefficient $\hat{b}_i$ is computed as $|b_i -
\hat{b}_i|/|b_i|$ when $b_i \neq 0$ (following \cite[Ex.
3.1]{Sastre18}), and as the absolute error $|b_i - \hat{b}_i|$
otherwise. This hybrid approach prevents numerical singularities
during the validation phase and allows the tool to identify stable
configurations for a broader class of matrix polynomials.

The main input and output parameters of the algorithm are detailed
in Table \ref{tab:algorithm_parameters}. To ensure proper execution,
the implementation validates the inputs and issues an error if the
leading coefficient is zero ($b_m = 0$), if the structural parameter
$s$ violates the required bounds ($s < 2$ or $4s > m$), or if the
specified \texttt{precision} is neither `single' nor `double'.
Nevertheless, the output variable \texttt{c\_vpa} provides the
coefficients computed with \texttt{ndigits} variable-precision
arithmetic, allowing users to manually round them to any arbitrary
target precision. Moreover, the output variable \texttt{all\_cvpa}
contains all sets of real and complex solutions obtained within the
same high-precision environment, enabling the manual selection of
alternative solution sets. Note that, by default, the parameter $s$
is chosen as the minimum value that achieves a $1M$ saving over the
PS method, thereby minimizing the number of intermediate matrix
powers stored and reducing memory overhead.

\begin{table}[htpb]
\centering \caption{Input and output parameters for Algorithm
\ref{MatrixPolEval1}.} \label{tab:algorithm_parameters}
\begin{tabularx}{\linewidth}{l X} 
\hline \multicolumn{2}{l}{\textit{Input Parameters}} \\ \hline
\texttt{b} & Vector of polynomial coefficients $\{b_0, \dots, b_m\}$ (default: Taylor expansion of $\exp$ of degree 8). \\
\texttt{precision} & Target IEEE arithmetic (default: `double', options: `single', `double'). \\
\texttt{type\_pol} & Structural variant selection (default: 1, options: 1, 2, or 3). \\
\texttt{ndigits} & Number of digits for VPA (default: 32 for double, 16 for single). \\
\texttt{s} & Maximum matrix power $A^s$ (default: minimal $s$ providing a $1M$ saving over the PS method to optimize memory usage). \\
\hline \hline \multicolumn{2}{l}{\textit{Output Parameters}} \\
\hline
\texttt{c\_prec} & Coefficients $c$ rounded to target IEEE \texttt{precision} maximizing stability. \\
\texttt{leading\_coeff\_sign} & Sign of $b_m$ ($\pm 1$). If $-1$, coefficients are computed for $-P_m(A)$. \\
\texttt{c\_vpa} & Stable coefficients evaluated in VPA (\texttt{ndigits}). \\
\texttt{er\_min} & Minimum relative error among maximum reconstruction errors. \\
\texttt{all\_cvpa} & Collection of all valid solution sets found by the solver. \\
\texttt{savings} & Number of matrix products saved compared to the optimal PS method. \\
\texttt{s}, \texttt{p} & Selected values for the structural
parameters $s$ and $p$. \\ \hline
\end{tabularx}
\end{table}

\begin{algorithm}[htpb]
\caption{Selection of stable matrix polynomial evaluation
coefficients based on the reconstruction error
test.}\label{MatrixPolEval1}
\begin{algorithmic}[1]
\Require \texttt{b}, \texttt{precision}, \texttt{type\_pol},
\texttt{s}, \texttt{ndigits} \Ensure \texttt{c\_prec},
\texttt{leading\_coeff\_sign}, \texttt{c\_vpa}, \texttt{er\_min},
\texttt{all\_cvpa}, \texttt{savings}, \texttt{s}, \texttt{p}

\If{$m < 8$ or $m \in \{9, 11\}$}
    \State Recommend PS method and exit.
\EndIf

\If{$b_m < 0$}
    \State \texttt{leading\_coeff\_sign} $\gets -1$
    \State Change the sign of the coefficients: \texttt{b} $\gets -$\texttt{b}.
\Else
    \State \texttt{leading\_coeff\_sign} $\gets 1$
\EndIf

\State If \texttt{s} is not provided, select the minimum $s$ based
on Theorem \ref{thm:cost} that saves $1M$ over the PS method.

\State Select the structural variant from
\eqref{type1_math}--\eqref{type3_math} based on \texttt{type\_pol}.

\State Initialize symbolic environment with \texttt{ndigits}
precision.

\State Construct the system of $4s+1$ nonlinear equations from
$y_{1s}(A) = P_m(A)$.

\State Solve the system using \texttt{vpasolve} to obtain all
possible solution sets \texttt{all\_cvpa}.

\State Round each real solution set to the target
\texttt{precision}.

\State For each real set, reconstruct the coefficients $\hat{b}_i$
and compute the stability metric $\epsilon_{max} = \max_i \{
\delta_i \}$, where: \Statex \quad $\delta_i = |b_i -
\hat{b}_i|/|b_i|$ if $b_i \neq 0$ \cite[Ex. 3.1]{Sastre18}, and
$\delta_i = |b_i - \hat{b}_i|$ otherwise.

\State Select the set that minimizes $\epsilon_{max}$ to define
\texttt{c\_prec} and \texttt{c\_vpa}, and set \texttt{er\_min}
$\gets \min\{\epsilon_{max}\}$.

\If{\texttt{er\_min} $> 10u$}
    \State Issue a warning regarding potential inaccuracies.
\Else
    \State Issue a notification indicating the formulas are stable.
\EndIf

\State Calculate \texttt{savings} $= C_{PS}(m) - C_{z1ps}(m)$.

\State \Return \texttt{c\_prec}, \texttt{leading\_coeff\_sign},
\texttt{c\_vpa}, \texttt{er\_min}, \texttt{all\_cvpa},
\texttt{savings}, \texttt{s}, \texttt{p}
\end{algorithmic}
\end{algorithm}

A MATLAB implementation of Algorithm \ref{MatrixPolEval1} is
available at \url{https://github.com/hipersc/MatrixPolynomials/}.
The suffix ``1'' in the function name denotes the reduction of one
matrix product relative to the PS method. To facilitate ease of use,
the function leverages the aforementioned default parameters,
allowing for a simplified call such as \texttt{c =
MatrixPolEval1(b)}. This minimal syntax automatically returns the
optimal evaluation coefficients in IEEE double precision for the
polynomial defined by the vector \texttt{b}. Further details
regarding advanced configurations and outputs are available in the
function's built-in documentation.

\section{Numerical results}\label{sec:numericalresults}

In this section, we present several numerical examples to
demonstrate the effectiveness and flexibility of the proposed
approach. All numerical experiments were conducted using MATLAB
R2024a on an Apple MacBook Pro equipped with an M4 Max chip and 48
GB of RAM.

\newpage
\begin{example}{Parameter Selection Flexibility.} \label{ex32Sastre18}

In \cite[Ex. 3.2, p. 243]{Sastre18}, an example was presented to
highlight the critical importance of selecting a stable set of
coefficients. This case involved evaluating the Taylor approximation
of the matrix exponential of degree $m=28$ using the evaluation
scheme defined in \eqref{msy0order1} and \eqref{msy1order1} with
$s=7$. According to \eqref{costz1s} and \eqref{PScost}, the
computational cost of this scheme is $1M$ lower than that of the PS
method.

However, it was found that a solution with multiplicity 10 existed,
leading to significant numerical instability if that specific
coefficient set was used. Evaluating the scalar case $\exp(x)$ for
$x=1$ served as a clear demonstration of this instability, yielding
a relative error of $|\exp(x)-y_{1,s=7}(x)|/\exp(x) = 9.81 \times
10^{-5} \gg u$ in IEEE double precision arithmetic. Furthermore, the
stability check for the best available solution was
$\mathcal{O}(10^{-15}) > 10u$. This is confirmed by the function
call

\texttt{c = MatrixPolEval1(1./factorial(sym(0:28)), [], [], [], 7)}

which issues a warning that the evaluation formulas are likely to be
inaccurate.

By contrast, using the default call

\texttt{c = MatrixPolEval1(1./factorial(sym(0:28)))}

\noindent the function automatically proposes the evaluation formula
based on \eqref{mixps}, \eqref{msy0order1}, and \eqref{msy1order1}
with $s=4$ and $p=12$. This selection results in an \texttt{er\_min}
$= 8.63 \times 10^{-17}$. When evaluated using 200-digit variable
precision arithmetic, the relative error is
$|\exp(1)-z_{1,p=12,s=4}(1)|/\exp(1) \approx 5.32 \times 10^{-17} <
u$.

If the stability check for these formulas fails, one of the
alternative selections for \texttt{type\_pol} can be chosen. If
those alternatives also prove unstable, other values of $s$ that
still achieve a $1M$ saving can be explored. For instance, setting
$s=5$ in the following call:
\begin{verbatim}
[c, ~, c_vpa, er_min, all_cvpa, savings, s, p] = ...
    MatrixPolEval1(1./factorial(sym(0:28)), [], [], [], 5)
\end{verbatim}
yields $p=8$ and \texttt{savings}=1. The \texttt{all\_cvpa} variable
contains 8 real solutions, with \texttt{er\_min} $= 8.63 \times
10^{-17}$, maintaining a relative error for $\exp(1)$ of $5.32
\times 10^{-17} < u$.

Similar values for \texttt{er\_min} and the relative error with
respect to $\exp(1)$ are obtained by selecting $s=6$ and $p=4$,
which also provides \texttt{savings}=1 and 20 sets of coefficients
(of which 4 are real). This example confirms the high degree of
flexibility offered by the evaluation formulas in finding stable,
cost-efficient parameters.
\end{example}

\begin{example}{Real Solutions and Stability Check for the Matrix Exponential.}\label{ex:realsol}

The stability of using \eqref{mixps} with \eqref{msy0order1} and
\eqref{msy1order1} to evaluate Taylor approximations of the matrix
exponential for degrees $m \in \{8, 12, 16, \dots, 81\}$ was
analyzed in \cite[Ex. 3.2]{Sastre18}. In that study, the evaluation
formulas $z_{1ps}$, $y_{1s}$, and $y_{0s}$ were defined by
\eqref{mixps}, \eqref{msy1order1}, and \eqref{msy0order1},
respectively. It was demonstrated that real and stable coefficient
sets exist for every value of $m$ according to the stability
criteria established therein.

The following MATLAB code uses the function \texttt{MatrixPolEval1}
to obtain the coefficients for all Taylor degrees $m \in \{8, 10,
12, 13, \dots, 81\}$ where a saving of one matrix product ($1M$) is
achieved relative to the PS method:

\begin{verbatim}
k = 0;
for m = [8 10 12:81]
    k = k + 1;
    b = 1./factorial(sym(0:m));
    % Get coefficients and stability check value
    [c, ~, ~, er_min, ~, ~, s, p] = MatrixPolEval1(b);
    er_minvec(k) = er_min;
end
er_min_max = max(er_minvec)
\end{verbatim}

The execution of this code confirms that a real, stable set of
coefficients exists for the evaluation scheme \eqref{msy0order1},
\eqref{msy1order1} and \eqref{mixps} across all considered degrees.
The maximum value of \texttt{er\_min} was $2.07 \times 10^{-16}$,
occurring at $m=46$ with parameters $s=6$ and $p=22$.

To further verify stability, consider a scalar $x > 0$ such that the
first term of the Taylor remainder equals the unit roundoff $u =
2^{-53}$ in double precision; i.e., $x^{47}/47! = 2^{-53}$, which
yields $x \approx 8.40$. Using the corresponding evaluation formula
for this value, the relative error was $|\exp(x) -
z_{1,p=22,s=6}(x)|/\exp(x) = 1.40 \times 10^{-17} < u$. In this
example, the results confirm that the proposed coefficients maintain
numerical accuracy even as the magnitude of $x$ increases.
\end{example}

\begin{example}{Applications to Finite Geometric Series, Control Theory, and Network Analysis.}\label{control}

Consider the finite matrix geometric series:
\begin{equation}\label{psiWestreich}
\Psi(N,A) = I + A + A^2 + \dots + A^{N-1}.
\end{equation}
This polynomial is fundamental in multirate sampled-data systems
\cite{89Wes, 92LeiNak} and in the analysis of large-scale
discrete-time control systems, particularly for computing
reachability properties and finite-time gramians \cite{Antoulas05}.
In network analysis, it is used to evaluate path connectivity and
the Katz centrality index, where it accounts for the total influence
across all paths of length up to $N-1$ \cite{Katz53}. Furthermore,
$\Psi(N,A)$ serves as a primary computational block for high-order
Schulz-type iterative methods used in the numerical computation of
generalized inverses \cite{Li11}.

From \cite[Table 1]{89Wes}, the cost of evaluating $\Psi(17,A)$
(degree $m=16$) is $6M$. While \cite[Sec. III]{92LeiNak} proposes an
improved algorithm for $N \leq 36$, its computational cost remains
$6M$. Using the proposed tool via the call
\texttt{c~=~MatrixPolEval1(ones(17,1))}, a set of coefficients in
IEEE double precision is obtained for evaluating $\Psi(17,A)$ using
the scheme in \eqref{mixps} with parameters $s=4$ and $p=0$. This
reduces the cost to $5M$ (a saving of $1M$). The stability check
yields $E_r = 3.69 \times 10^{-16} < 10u$, indicating that the
formulas are numerically stable and likely to be accurate.

Note that the call

\texttt{[c,\textasciitilde, c\_vpa, er\_min, all\_cvpa] =
MatrixPolEval1(ones(17,1))}

\noindent identifies 4 sets of real coefficients and 8 sets of
complex coefficients. The function selects the most stable real set
by default and confirms its accuracy with a stability check of
\texttt{er\_min} $\leq 10u$.

For comparison, the evaluation schemes proposed in \cite[Sec.
II]{89Wes} for $N \in \{9, 13, 17\}$ are:
\begin{subequations}
\begin{align}
\label{Psi9Wes} \Psi(9,A)  &= I + (I + A^2)(A + A^2)(I + A^4), \\
\label{Psi13Wes} \Psi(13,A) &= I + (I + A^2)(A + A^2)(I + A^4 + A^8), \\
\label{Psi17Wes} \Psi(17,A) &= I + (I + A^2)(A + A^2)(I + A^4)(I +
A^8),
\end{align}
\end{subequations}
where matrix powers are computed once and reused. The costs for
these evaluations are $4M$, $5M$, and $6M$, respectively. In
contrast, the scheme proposed by \texttt{MatrixPolEval1} using
\eqref{msy0order1}, \eqref{msy1order1} and \eqref{mixps} reduces the
cost by $1M$ in each case.

We compared both methods for $\Psi(17,A)$ using 100 random $100
\times 100$ matrices with the following MATLAB code:
\begin{verbatim}
rng(0); % Ensure reproducibility
c = MatrixPolEval1(ones(17,1));
for i = 1:100
    A = rand(100);
    I = eye(size(A)); A2 = A^2; A4 = A2^2; A8 = A4^2;
    % Westreich evaluation
    W = I + (I + A2) * (A + A2) * (I + A4) * (I + A8);
    % Proposed evaluation (s=4, p=0)
    A3 = A2 * A;
    y0 = A4 * (c(1)*A4 + c(2)*A3 + c(3)*A2 + c(4)*A);
    y1 = (y0 + c(5)*A4 + c(6)*A3 + c(7)*A2 + c(8)*A) ...
         * (y0 + c(9)*A4 + c(10)*A3 + c(11)*A2) ...
         + c(12)*y0 + c(13)*A4 + c(14)*A3 ...
         + c(15)*A2 + c(16)*A + c(17)*I;
    er_rel(i) = norm(y1 - W, 1) / norm(W, 1);
end
max_error = max(er_rel)
\end{verbatim}

The results showed a maximum relative difference of $4.00 \times
10^{-16} \approx 3.61u$. For $10000 \times 10000$ matrices, the
difference was $8.77 \times 10^{-17} < u$. These comparisons are
summarized in Table \ref{Tab_compPSWy1}, demonstrating high accuracy
and consistency with the stability check results from \cite[Ex.
3.1]{Sastre18}.

\begin{table}[ht!]
\centering \setlength{\tabcolsep}{10pt}
\begin{tabular}{lccc}
\hline Matrix Size & \multicolumn{3}{c}{Rel. Diff. (multiples of
$u$)} \\ \cline{2-4} ($n$) & Proposed vs PS & Proposed vs W & W vs
PS \\ \hline
100   & $3.28u$ & $3.61u$ & $3.30u$ \\
1000  & $2.03u$ & $2.01u$ & $2.06u$ \\
10000 & $1.92u$ & $0.76u$ & $1.91u$ \\ \hline
\end{tabular}
\caption{Relative differences for evaluating $\Psi(17,A)$ for 100
random matrices of sizes $n=100, 1000,$ and $10000$. Results compare
the matrix polynomial evaluation using the coefficients from
\texttt{MatrixPolEval1}, the PS method \eqref{PPS} with
$(m,s)=(16,4)$, and the Westreich (W) formula \eqref{Psi17Wes}.
Values are expressed as multiples of the unit roundoff in IEEE
double precision arithmetic $u = 2^{-53} \approx 1.11 \times
10^{-16}$.} \label{Tab_compPSWy1}
\end{table}

{\sloppy Finally, we tested matrices from the Matrix Computation
Toolbox (MCT)~\cite{MCT} and the Eigtool MATLAB Package (EMP)
\cite{EMP} with sizes $n = 100,$ $500$ and $1000$. The reference
``exact'' values for $\Psi(N,A)$ with $N = 9,$ $13$ and $17$ were
computed using variable precision arithmetic (VPA) via the MATLAB
Symbolic Math Toolbox. Specifically, each matrix $A$ was converted
to a symbolic representation, and the summation $\sum_{i=0}^{N-1}
A^i$ was evaluated using the PS formula \eqref{PPS} with $s=4$ for
degrees $m = 8, 12$ and $16$ with 100 decimal digits of accuracy.
Matrices 16, 17, 21, 42, and 44 from the MCT and matrix 4 from the
EMP gave infinite or Nan results by following that procedure and
where eliminated from the tests.\par}

Figures \ref{fig_m8}, \ref{fig_m12}, and \ref{fig_m16} show the
sorted 1-norm relative errors. While Westreich's formulas
occasionally show slightly different error norms, the PS method and
the proposed \texttt{MatrixPolEval1} scheme yield virtually
identical accuracy relative to the high-precision VPA baseline.
\end{example}

\begin{figure}[H]
    \begin{tabular}{cccccc}
        \includegraphics[width=6cm]{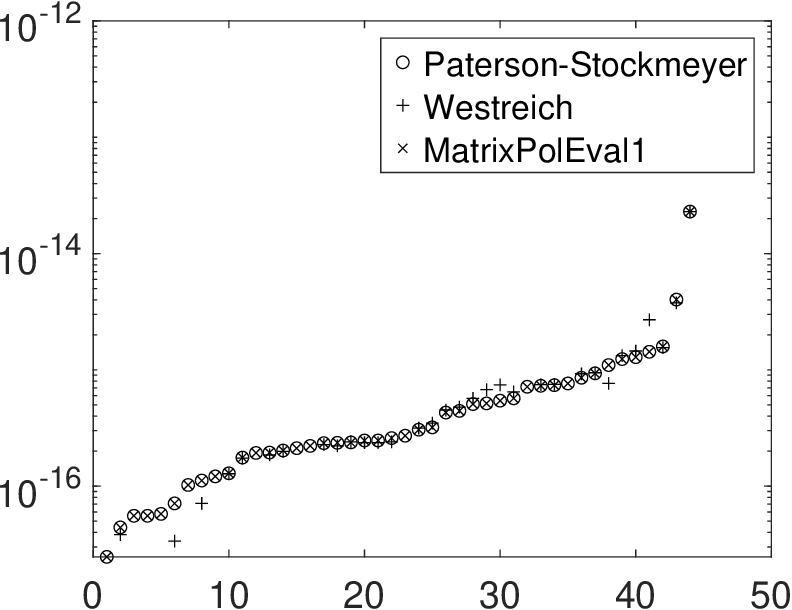}
        &\includegraphics[width=6cm]{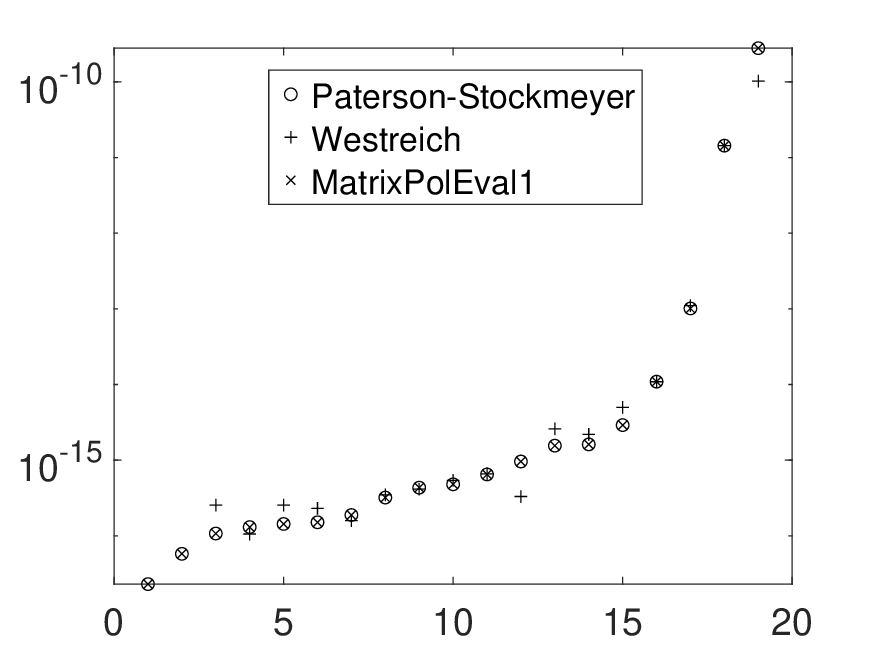}\\
        \includegraphics[width=6cm]{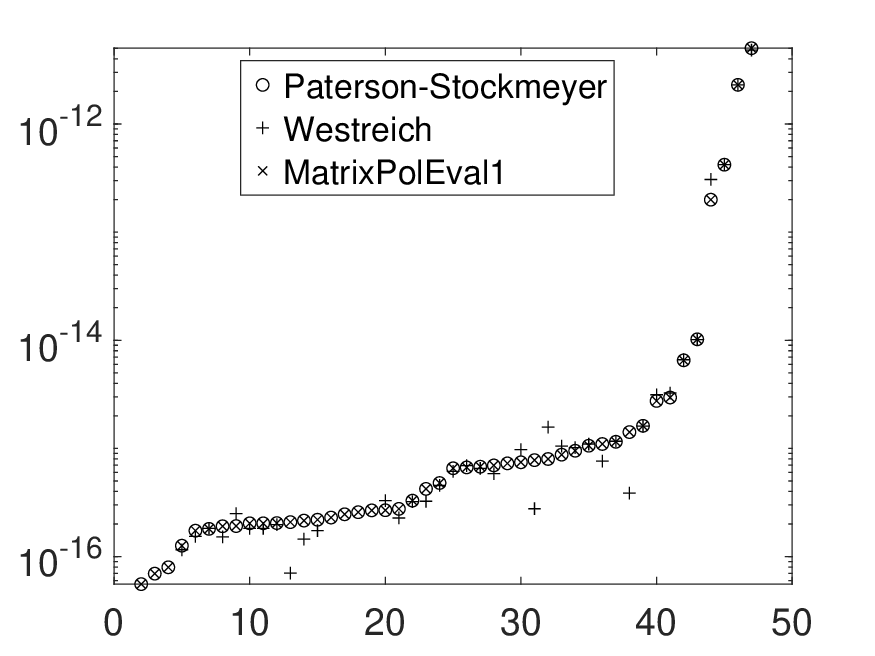}
        &\includegraphics[width=6cm]{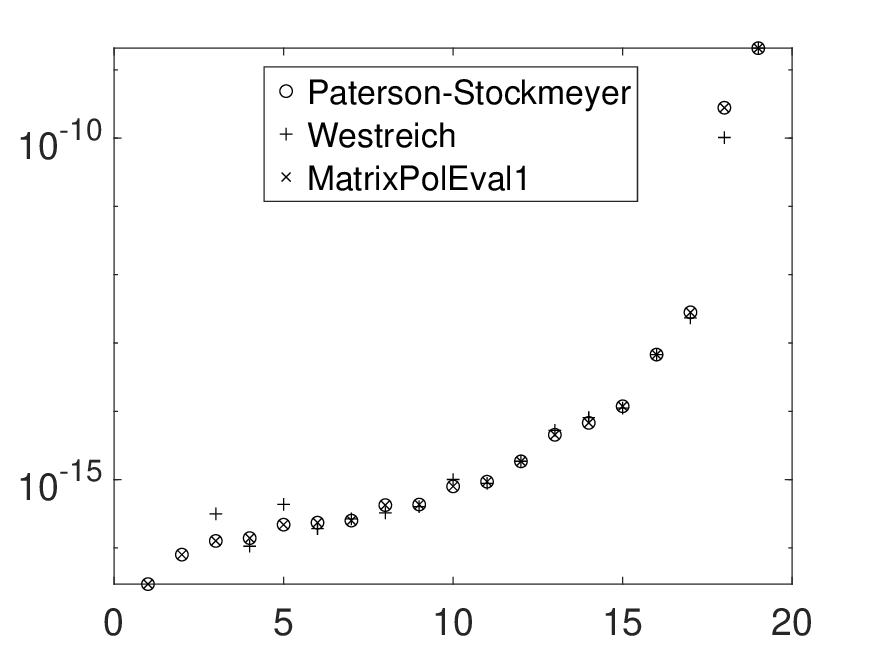}\\
        \includegraphics[width=6cm]{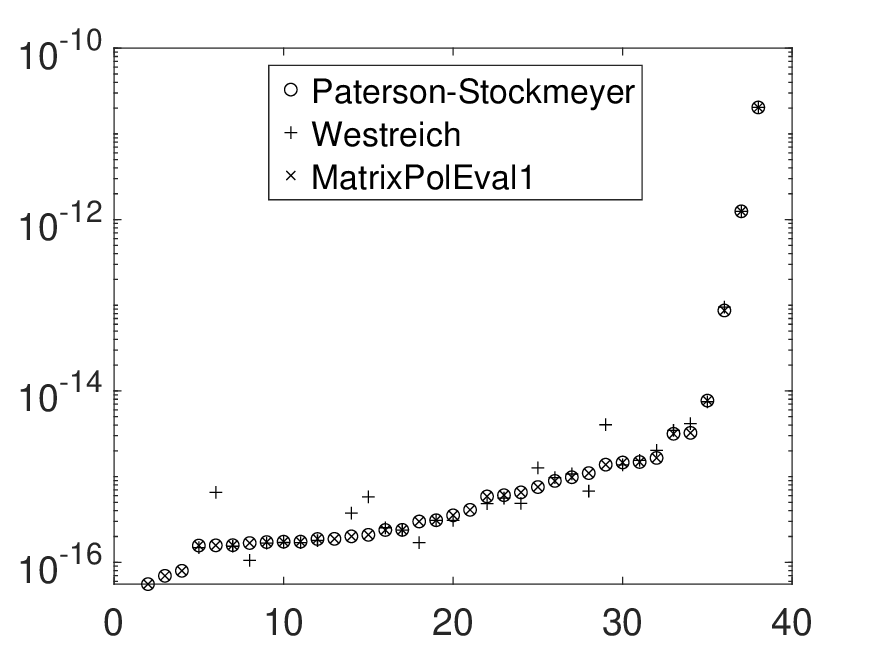}
        &\includegraphics[width=6cm]{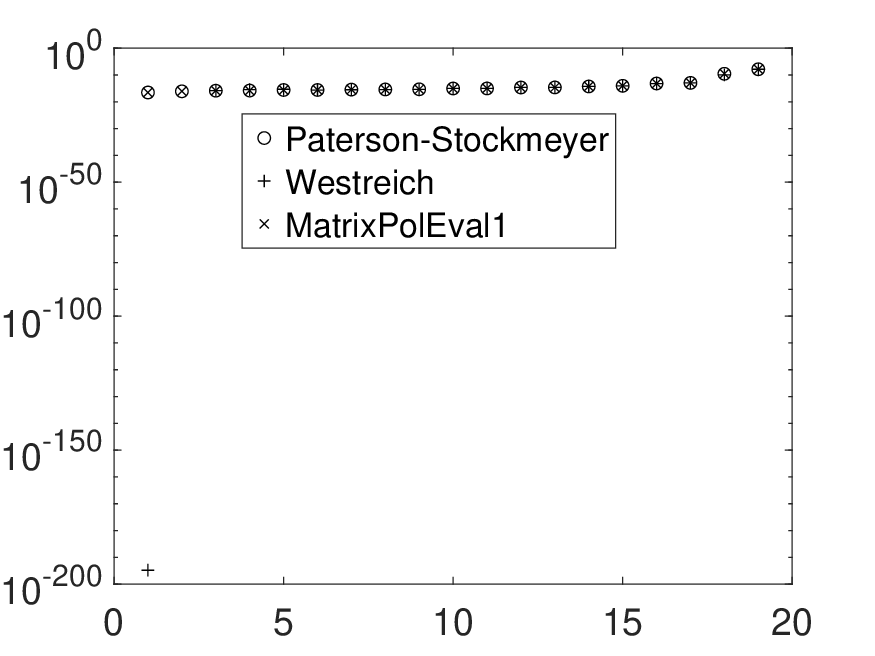}\\
        (\textbf{a}) MCT. & (\textbf{b}) EMP.\\
    \end{tabular}
    \caption{Relative error comparison for the evaluation of $\Psi(9,A)$. Plots show sorted 1-norm relative errors in IEEE double precision for matrix dimensions $n=100$ (top), $n=500$ (middle), and $n=1000$ (bottom). The methods compared are the Paterson--Stockmeyer (PS) formula \eqref{PPS}, Westreich's approach \eqref{Psi9Wes} \cite{89Wes}, and the optimized evaluation formula obtained via \texttt{MatrixPolEval1}. Left-hand plots correspond to MCT test matrices, while right-hand plots correspond to EMP
    matrices.}
  \label{fig_m8}

\end{figure}

\begin{figure}[H]
    \begin{tabular}{cccccc}
        \includegraphics[width=6cm]{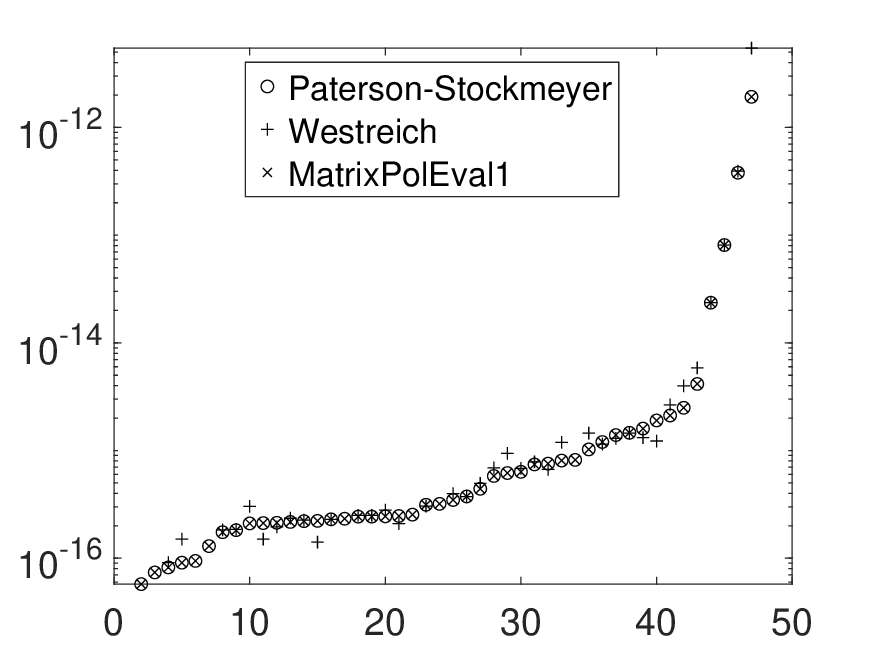}
        &\includegraphics[width=6cm]{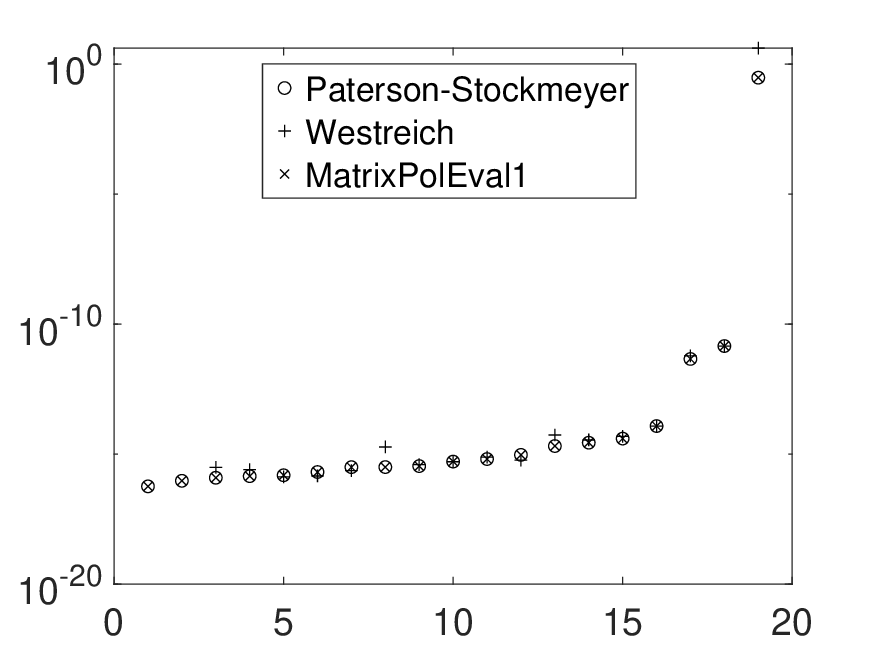}\\
        \includegraphics[width=6cm]{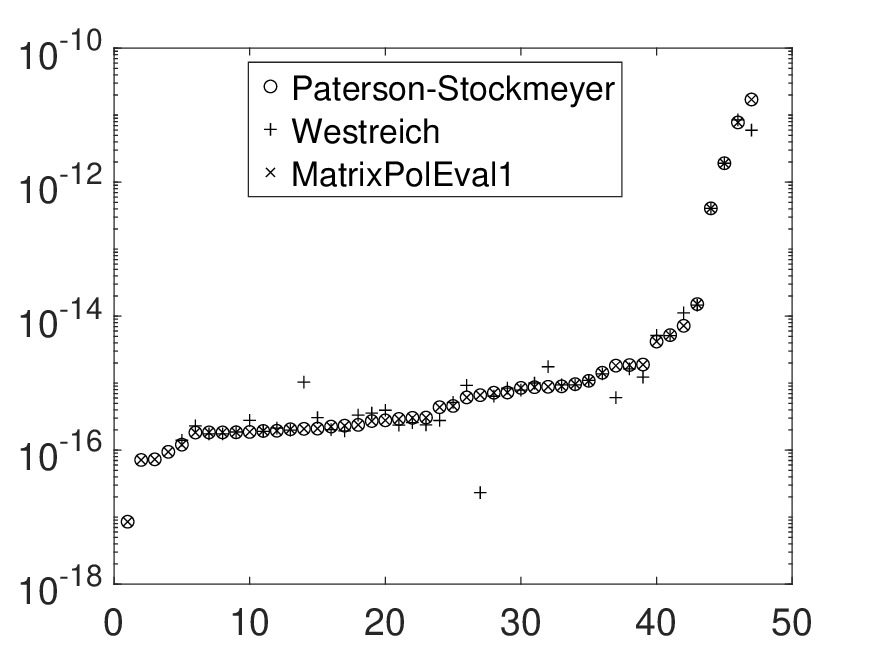}
        &\includegraphics[width=6cm]{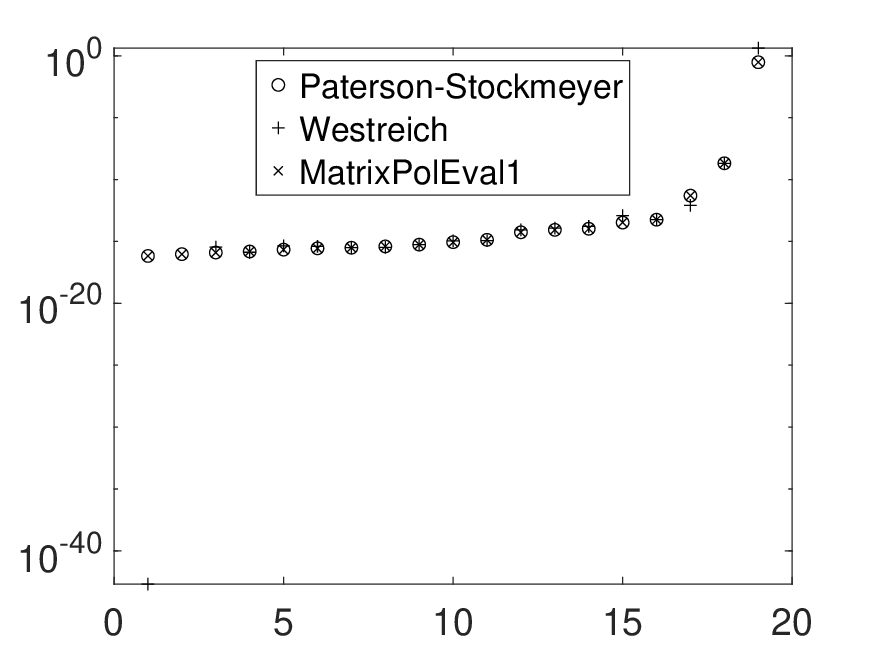}\\
        \includegraphics[width=6cm]{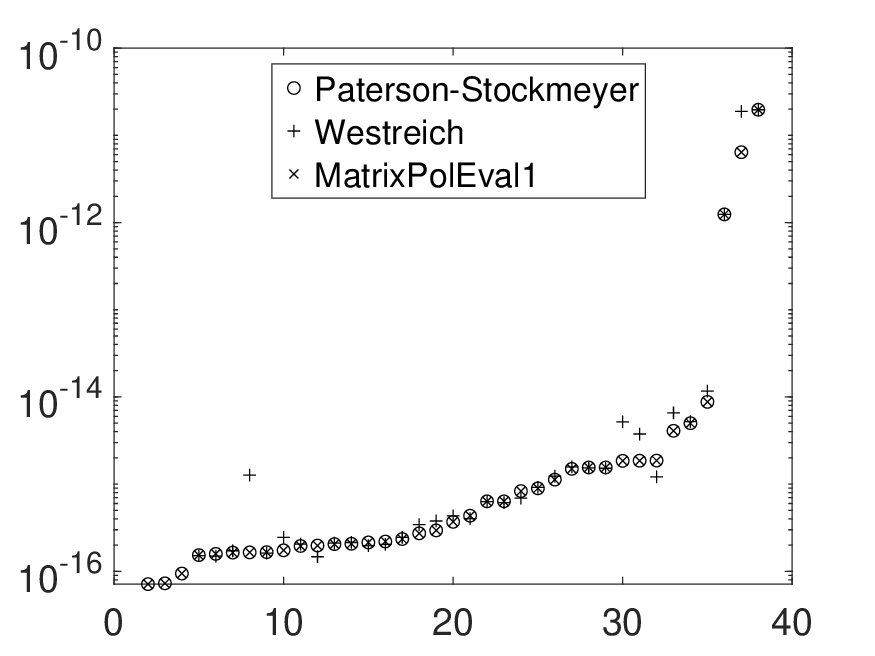}
        &\includegraphics[width=6cm]{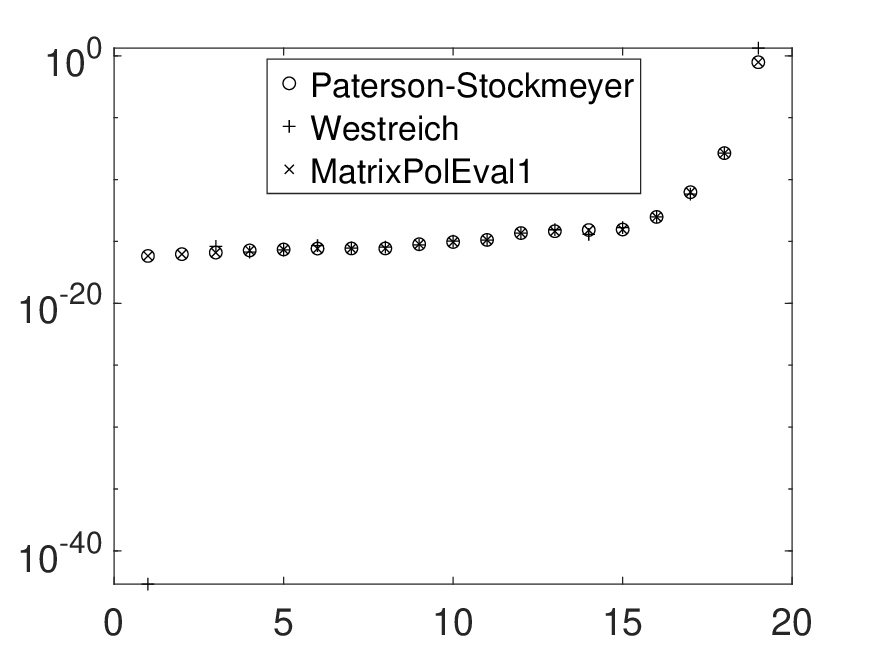}\\
        (\textbf{a}) MCT. & (\textbf{b}) EMP.\\
    \end{tabular}
    \caption{Relative error comparison for the evaluation of $\Psi(13,A)$. Plots show sorted 1-norm relative errors in IEEE double precision for matrix dimensions $n=100$ (top), $n=500$ (middle), and $n=1000$ (bottom). The methods compared are the PS formula \eqref{PPS}, Westreich's approach \eqref{Psi13Wes} \cite{89Wes}, and the optimized evaluation formula obtained via \texttt{MatrixPolEval1}. Left-hand plots correspond to MCT test matrices, while right-hand plots correspond to EMP
    matrices.}
    \label{fig_m12}
\end{figure}

\begin{figure}[H]
    \begin{tabular}{cccccc}
        \includegraphics[width=6cm]{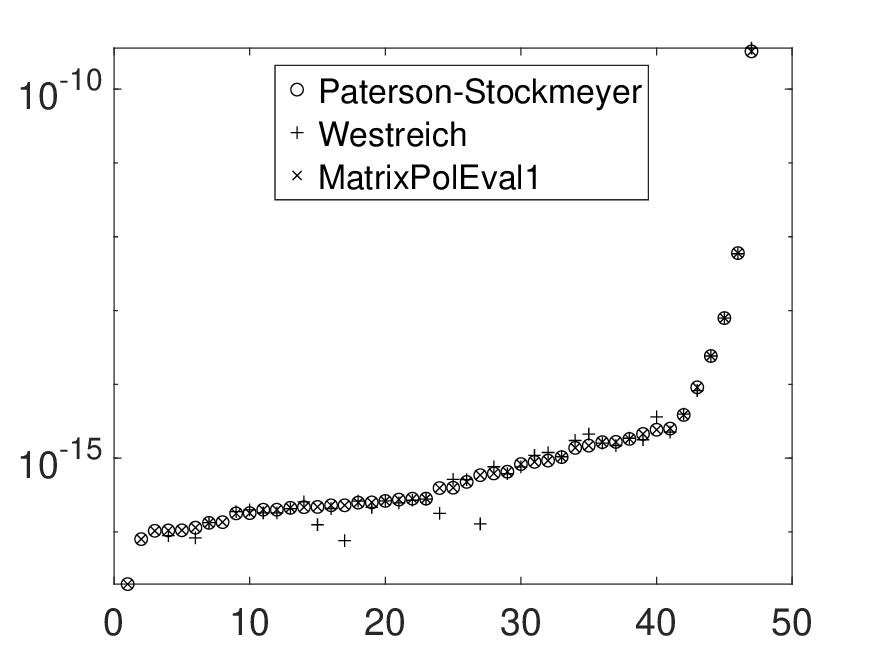}
        &\includegraphics[width=6cm]{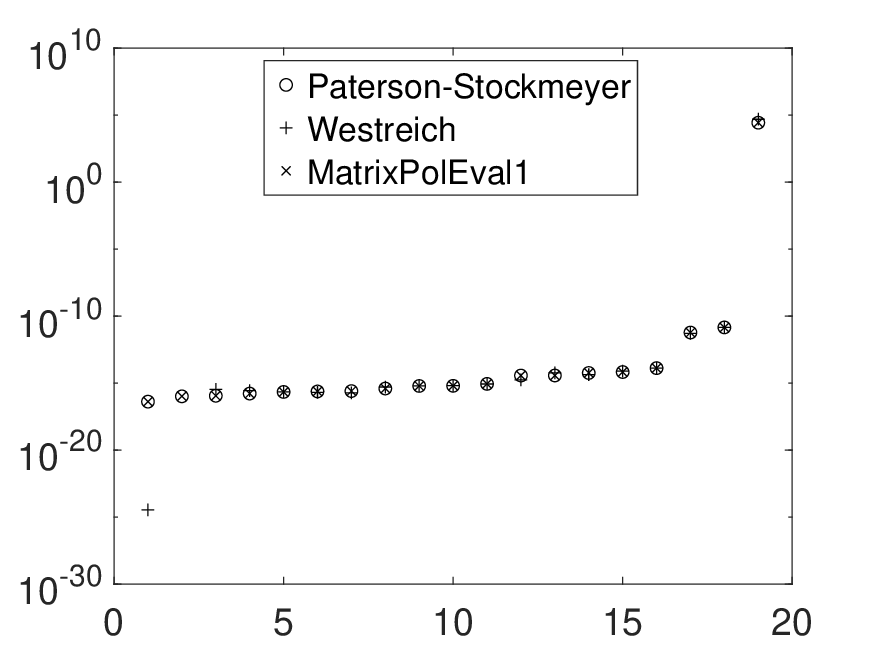}\\
        \includegraphics[width=6cm]{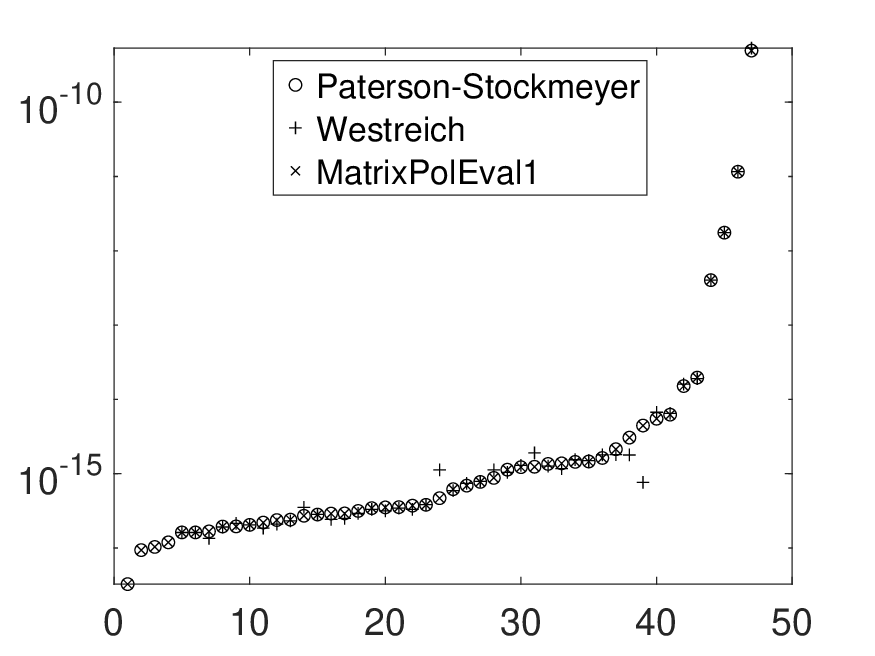}
        &\includegraphics[width=6cm]{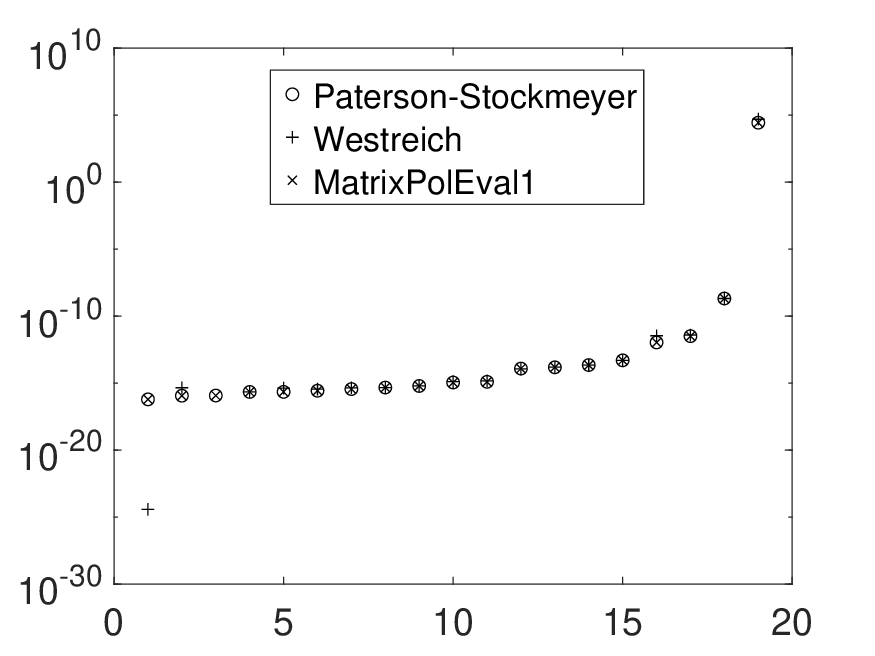}\\
        \includegraphics[width=6cm]{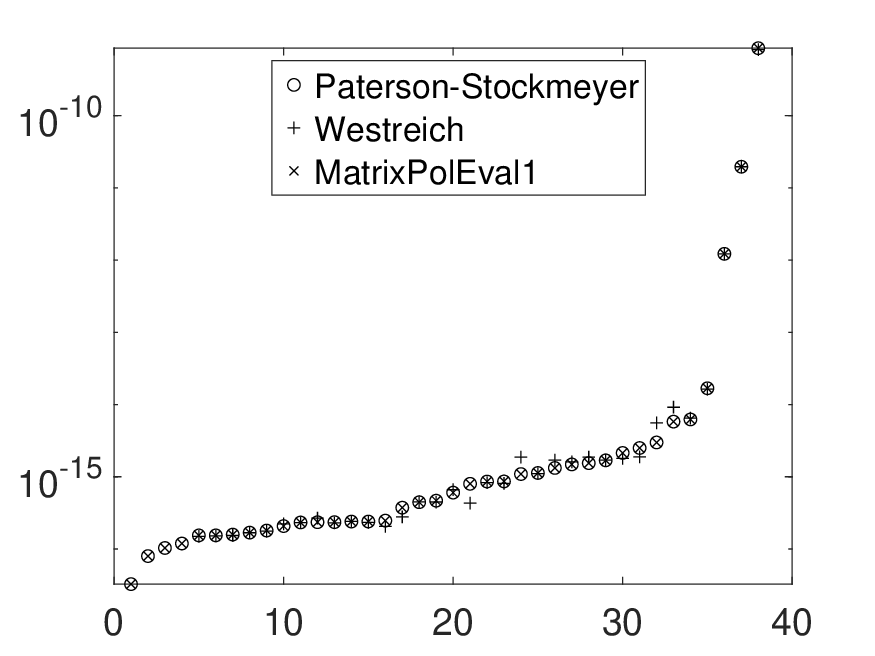}
        &\includegraphics[width=6cm]{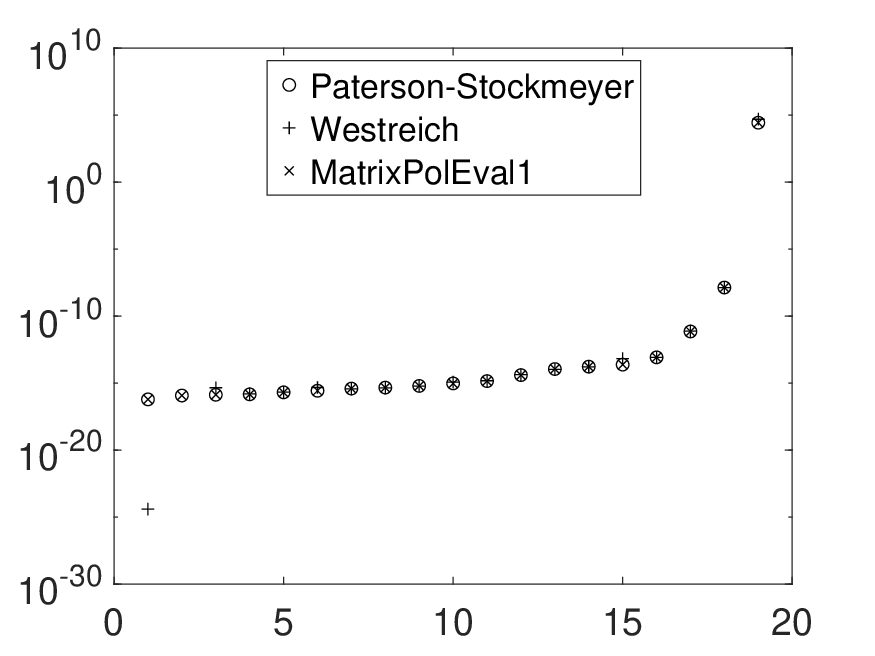}\\
        (\textbf{a}) MCT. & (\textbf{b}) EMP.\\
    \end{tabular}
    \caption{Relative error comparison for the evaluation of $\Psi(17,A)$. Plots show sorted 1-norm relative errors in IEEE double precision for matrix dimensions $n=100$ (top), $n=500$ (middle), and $n=1000$ (bottom). The methods compared are the PS formula \eqref{PPS}, Westreich's approach \eqref{Psi17Wes} \cite{89Wes}, and the optimized evaluation formula obtained via \texttt{MatrixPolEval1}. Left-hand plots correspond to MCT test matrices, while right-hand plots correspond to EMP
    matrices.}
    \label{fig_m16}
\end{figure}

\section{Conclusions}\label{sec:conclusions}

This work has introduced a framework and the corresponding MATLAB
tool, \texttt{MatrixPolEval1}, designed to optimize the evaluation
of matrix polynomials. The proposed schemes consistently achieve a
reduction in computational cost by saving one matrix product ($1M$)
compared to the widely used PS method. This methodology is general
and can be applied to evaluate matrix polynomials of degrees 8, 10,
and $\geq 12$ with non-zero leading coefficients, providing a more
efficient alternative to traditional techniques across various
mathematical and engineering applications.

One of the significant advantages of the proposed method is its
ability to identify stable sets of coefficients from among all
potential solution sets. The tool incorporates a mechanism to warn
the user if no stable solution sets are found, and its parameter
flexibility allows for the discovery of stable solutions in cases
where other configurations may yield unstable results.

For specific applications such as the finite geometric series
$\Psi(N,A)$, the proposed method provides a competitive alternative
to classical specialized algorithms, such as the approach by
Westreich, while preserving the required numerical precision.
Extensive testing using the Matrix Computation Toolbox (MCT) and the
Eigtool MATLAB Package (EMP) confirms that the stability check
parameter, \texttt{er\_min}, is a reliable predictor of the actual
numerical performance. This behavior is consistent with previous
results, where the same stability criterion was successfully applied
to the evaluation of the matrix exponential, the matrix cosine, and
other related matrix functions \cite{SID19, SIAPD19}. Beyond these
specific cases, the tool is intended as a general-purpose utility
for the stable evaluation of arbitrary matrix polynomials with
reduced computational cost. Future work will focus on providing new
tools for matrix polynomial evaluation capable of achieving greater
computational savings.

\section*{CRediT authorship contribution statement}

\textbf{Jose Miguel Alonso:} Software, Investigation, Validation,
Data curation, Writing - review \& editing.

\textbf{Jorge Sastre:} Conceptualization, Methodology, Software,
Formal analysis, Investigation, Resources, Writing - original draft,
Writing - review \& editing, Supervision, Project administration,
Funding acquisition.

\textbf{Javier Ib\'a\~nez:} Investigation, Validation, Formal
analysis, Writing - review \& editing.

\textbf{Emilio Defez:} Formal analysis, Validation, Writing - review
\& editing.

\section*{Declaration of Generative AI and AI-assisted technologies in the writing process}

During the preparation of this work the author(s) used Gemini in
order to improve the language, flow, and technical clarity of the
manuscript. After using this tool, the authors reviewed and edited
the content as needed and take full responsibility for the technical
accuracy and the final content of the publication.

\section*{Declaration of competing interest}

The authors declare that they have no known competing financial
interests or personal relationships that could have appeared to
influence the work reported in this paper.

\section*{Acknowledgements}\label{Ack}
This work has been supported by the Generalitat Valenciana Grant \\
CIAICO/2023/270.


\end{document}